\documentclass{article}%
\usepackage{amssymb}
\usepackage{amsfonts}
\usepackage{amsmath}
\usepackage{graphicx}%
\providecommand{\U}[1]{\protect\rule{.1in}{.1in}}
\newtheorem{theorem}{Theorem}

\newtheorem{corollary}[theorem]{Corollary}

\newtheorem{definition}[theorem]{Definition}
\newtheorem{example}[theorem]{Example}

\newtheorem{lemma}[theorem]{Lemma}

\newtheorem{proposition}[theorem]{Proposition}
\newtheorem{remark}[theorem]{Remark}

\newenvironment{proof}[1][Proof]{\noindent\textbf{#1.} }{\ \rule{0.5em}{0.5em}}
\begin{document}

\title{A characterization of $\mu-$equicontinuity for topological dynamical systems}
\author{Felipe Garc\'{\i}a-Ramos\\felipegra@math.ubc.ca\\University of British Columbia}
\maketitle

\begin{abstract}
Two different notions of measure theoretical equicontinuity ($\mu
-$equicontinuity) for topological dynamical systems with respect to Borel
probability measures appeared in \cite{Gilman1} and \ \cite{mtequicontinuity}.
We show that if the probability space satisfies Lebesgue's density theorem and
Vitali's covering theorem (for example a Cantor set or a subset of
$\mathbb{R}^{d}$) then both notions are equivalent. To show this we
characterize Lusin measurable maps using $\mu-$continuity points. As a
corollary we obtain new characterizations of $\mu-$pairwise sensitivity
($\mu-$sensitivity) for ergodic TDS. We also introduce a subfamily of $\mu
-$equicontinuous systems defined by locally periodic behaviour.

\end{abstract}
\tableofcontents

\section{Introduction}

A topological dynamical system (TDS), $(X,T)$, is a continuous transformation
$T$ on a compact metric space $X$. A class of rigid systems is the family of
equicontinuous TDS. A TDS is equicontinuous if the family $\{T^{i}\}$ is
equicontinuous. Equivalently we can define equicontinuity using equicontinuity
points: $x\in X$ is an equicontinuity point if \textit{for every} $y$ close to
$x$ we have that $T^{i}(x),T^{i}(y)$ stay close for all $i$. A TDS is
equicontinuous if and only if every point is equicontinuous. Sensitivity (or
sensitivity to initial conditions) is considered a weak form of chaos.
Akin-Auslander-Berg~\cite{AkinAuslander} showed that a minimal TDS is either
sensitive or equicontinuous. .

Equicontinuity is a very strong property and different attempts have been made
to weaken this property. We focus on the ones made with the use of Borel
probability measures.

While studying cellular automata (a subclass of TDS), Gilman \cite{Gilman1}%
\cite{Gilman2} introduced the concept of $\mu-$equicontinuity points: $x\in X$
is a $\mu-$equicontinuity point if for \textit{most }$y$ close to $x$ we have
that $T^{i}(x),T^{i}(y)$ stay close for all $i\in\mathbb{Z}_{+}$. He also
defined $\mu-$sensitivity and he showed that if $T$ is a cellular automaton
and $\mu$ a shift$-$ergodic measure (not necessarily invariant with respect to
$T$) then $T$ is either $\mu-$sensitive or the set of $\mu-$equicontinuity
points has measure one. With exactly the same approach one can define $\mu
-$equicontinuity points for TDS with respect to measures that are not
necessarily ergodic or invariant under any transformation.

In \cite{Cadre2005375}, Cadre-Jacob introduced $\mu-$paiwise sensitivity
($\mu-$sensitivity) for TDS. This notion was characterized for ergodic TDS by
Huang-Li-Ye in \cite{mtequicontinuity}. It was shown that an ergodic TDS is
either $\mu-$pairwise sensitive or $\mu-$equicontinutous; $(X,T)$ is $\mu
-$equicontinuous if for every $\varepsilon>0$ there exists a compact set $M$
such that $\mu(M)>1-\varepsilon$ and $T\shortmid_{M}$is equicontinuous. They
showed that every ergodic $\mu-$equicontinuous TDS has discrete spectrum.

We show that if $(X,\mu)$ satisfies Lebesgue's density theorem and Vitali's
covering theorem (for example when $X\subset\mathbb{R}^{d}$ or when $X$ is a
Cantor space) then $(X,T)$ is $\mu-$equicontinuous if and only if almost every
point is a $\mu-$equicontinuity point (Theorem \ref{3i1}). As a consequence we
get that if $\mu$ is ergodic then $(X,T)$ is $\mu-$pairwise sensitive if and
only if there are no $\mu-$equicontinuity points (Theorem \ref{musens}). We
also define a subclass of $\mu-$equicontinuous systems that rely on a local
periodicity notion and we prove they have discrete rational spectrum when
$\mu$ is ergodic.

In Section 2 we define and study topological and measure theoretical forms of
equicontinuity and local periodicity ($\mu-$equicontinuity and $\mu-LP$). To
characterize $\mu-$equicontinuity we study $\mu-$continuity, which is a
property of functions between metric measure spaces (not a dynamical concept).
In Section 3, we discuss $\mu-$equicontinuity with respect to invariant
measures; we show ergodic $\mu-LP$ TDS on Cantor spaces have rational spectrum
and we characterize $\mu-$sensitivity.

The characterization of $\mu-$continuity (Theorem \ref{tfae} ) also helps
characterize other weaker forms of measure theoretical equicontinuity. In
\cite{weakeq}, $\mu-$mean equicontinuity is introduced; Theorem \ref{tfae} (of
this paper) is applied to characterize $\mu-$mean equicontinuity (Theorem 21
in \cite{weakeq}), and it is shown that ergodic TDS are $\mu-$mean
equicontinuous if and only if they have discrete spectrum.

We remark that Theorem \ref{3i1} holds for any continuous semigroup action on
a compact metric space (in particular $\mathbb{Z}_{+}^{d}-$actions and
$\mathbb{R-}$flows). See comment after Theorem \ref{3i1}.

\section{Equicontinuity and local periodicity}

Throughout this paper $X$ will denote a compact metric space with metric
$d$.\ For $x\in X$ we represent the\textbf{\ balls centered in }$x$ with
$B_{n}(x)\mathbf{:}=\left\{  z\mid d(x,z)\leq1/n\right\}  $.

A \textbf{topological dynamical system (TDS)} is a pair $(X,T)$ where $X$ is a
compact metric space and $T:X\rightarrow X$ is a continuous transformation.

Two TDS $(X_{1},T_{2}),$ $(X_{2},T_{2})$ are \textbf{conjugate
(topologically)} if there exists a homeomorphism $f:X_{1}\rightarrow X_{2}$
such that $T_{2}\circ f=f\circ T_{1}$.

\bigskip An important kind of TDS are TDS on Cantor spaces.

The\textbf{\ }$n-$\textbf{window}, $W_{n}\subset$ $\mathbb{Z}_{+}$ is defined
as the interval of radius $n$ centred at the origin; a \textbf{window}%
\textit{\ }is an $n-$window for some $n$. Let $\mathcal{A}$ be a finite set.
For $W\subset\mathbb{Z}_{+}$ and $x\in X,$ $x_{W}\in$ $\mathcal{A}%
^{\mathbb{\ }\left\vert W\right\vert }$ denotes the restriction of $x$ to $W.$
The set $\mathcal{A}^{\mathbb{Z}_{+}}$ endowed with the Cantor metric ( given
by $d(x,y)=\frac{1}{m},$ where $m$ is the largest integer such that $x_{W_{m}%
}=y_{W_{m}}$) is called a \textbf{Cantor space}. It is known that Cantor
spaces are zero dimensional, i.e. there exists a countable base of clopen sets.

Another important class of systems are Euclidean TDS. We endow the set
$\mathbb{R}^{d}$ with the Euclidean metric. We say $(X,T)$ is
\textbf{Euclidean} if $X\subset\mathbb{R}^{d}.$

\subsection{Equicontinuity}

Mathematical definitions of chaos have been widely studied. Many of them
require the system to be sensitive to initial conditions. A TDS $(X,T)$ is
\textbf{sensitive (or sensitive to initial conditions)} if there exists
$\varepsilon>0$ such that for every open set $A\subset X$ there exist $x,y\in
A$ and $i\in$ $\mathbb{Z}_{+}$ such that $d(T^{i}x,T^{i}y)>\varepsilon.$ On
the other hand, equicontinuity represents predictable behaviour. As we
mentioned a TDS is equicontinuous if $\left\{  T^{i}\right\}  $ is an
equicontinuous family, and a minimal TDS is either sensitive or equicontinuous
\cite{AkinAuslander}.

\bigskip

We can also define equicontinuity by defining equicontinuity points.

\begin{definition}
Let $(X,T)$ be a TDS. We define the \textbf{orbit metric} $d_{o}$ on $X$ as
$d_{o}(x,y):=\sup_{i\geq0}\left\{  d(T^{i}x,T^{i}y)\right\}  ,$ and the
\textbf{orbit balls} as
\begin{align*}
B_{m}^{o}(x)  &  =\left\{  y\mid d_{o}(x,y)\leq1/m\right\} \\
&  =\left\{  y\mid d(T^{i}(x),T^{i}(y)\leq1/m\text{ }\forall i\in
\mathbb{N}\right\}  .
\end{align*}

Let $(X,T)$ be a TDS. A point $x$ is an \textbf{equicontinuity point} of $T$
if for every $m\in\mathbb{N}$ there exists $n\in\mathbb{N}$ such that
$B_{n}(x)\subset B_{m}^{o}(x)$.
\end{definition}

A TDS $(X,T)$ is equicontinuous if and only if every $x\in X$ is an
equicontinuity point. A TDS is equicontinuous if and only if it is uniformly
equicontinuous, i.e. for all $m\in\mathbb{N}$ there exists $n\in\mathbb{N}$
such that $B_{n}(x)\subset B_{m}^{o}(x)$ for every $x\in X.$

\begin{definition}
Let $M\subset X.$ We say $T\shortmid_{M}$is \textbf{equicontinuous} if for all
$x\in M$ and $m\in\mathbb{N}$ there exists $n\in\mathbb{N}$ such that
$B_{n}(x)\cap M\subset B_{m}^{o}(x).$
\end{definition}

\subsection{$\mu-$Equicontinuity}

We will use $\mu$ to denote Borel probability measures on $X$ ($\mu$ does not
need to be invariant under $T$).

Basedo on \cite{Gilman1} we define $\mu-$equicontinuity points.

\begin{definition}
A point $x$ is a $\mu-$\textbf{equicontinuity }point\textbf{\ }of $(X,T)$ if
for all $m\in\mathbb{N}$ we have%
\[
\lim_{n\rightarrow\infty}\frac{\mu(B_{m}^{o}(x)\cap B_{n}(x))}{\mu(B_{n}%
(x))}=1.
\]

\end{definition}

If $x$ is an equicontinuity point in the support of $\mu$ then $x$ is a $\mu
-$equicontinuity point.

The following definition appeared in \cite{mtequicontinuity}.

\begin{definition}
A TDS $(X,T)$ is $\mu-$\textbf{equicontinuous} if for every $\varepsilon>0$
there exists a compact set $M$ such that $\mu(M)>1-\varepsilon$ and
$T\shortmid_{M}$is equicontinuous.
\end{definition}

If $(X,T)$ is an equicontinuous TDS and $\mu$ a Borel probability measure then
$(X,T)$ is $\mu-$equicontinuous. In \cite{mueqca} it was shown that there
exists a sensitive TDS on $\left\{  0,1\right\}  ^{\mathbb{Z}}$ such that
$(X,T)$ is $\mu-$equicontinuous for every ergodic Markov chain $\mu.$

\bigskip

\begin{definition}
We say $(X,\mu)$ satisfies \textbf{Lebesgue's density theorem} if for every
Borel set $A$ we have that%
\begin{equation}
\lim_{n\rightarrow\infty}\frac{\mu(A\cap B_{n}(x))}{\mu(B_{n}(x))}=1\text{
\ for a.e. }x\in A.
\end{equation}

\end{definition}

The original Lebesgue's density theorem applies to $\mathbb{R}^{d}$ and the
Lebesgue measure. If $X\subset\mathbb{R}^{d}$ and $\mu$ is a Borel probability
measure then $(X,\mu)$ satisfies Lebesgue's density theorem ( see Remark 2.4
in \cite{kaenmaki2010local}).

\begin{theorem}
[Levy's zero-one law \cite{durrett2010probability} pg.262]\label{levy}\bigskip
Let $\Sigma$ be a sigma-algebra on a set $\Omega$ and $P$ a probability
measure. Let $\left\{  \mathcal{F}_{n}\right\}  \subset\Sigma$ be a filtration
of sigma-algebras, that is, a sequence of sigma-algebras $\left\{
\mathcal{F}_{n}\right\}  $, such that $\mathcal{F}_{i}\subset\mathcal{F}%
_{i+1}$ for all $i$. If $D$ is an event in $\mathcal{F}_{\infty}$ (the
smallest sigma-algebra that contains every $\mathcal{F}_{n}$) then
$P(D\mid\mathcal{F}_{n})\rightarrow1_{D}$ almost surely.
\end{theorem}

\begin{corollary}
\label{lebesgue}Let $X$ be a Cantor space and $\mu$ a Borel probability
measure. Then $(X,\mu)$ satisfies Lebesgue's density theorem.
\end{corollary}

\begin{proof}
Let $n\in\mathbb{N}$, and $\mathcal{F}_{n}$ the smallest sigma-algebra that
contains all the balls $\left\{  B_{n^{\prime}}(x)\right\}  $ with $n^{\prime
}\leq n$ $.$ It is easy to see that $\mathcal{F}_{\infty}$ is the Borel
sigma-algebra on $X.$ The desired result is a direct application from Theorem
\ref{levy}.
\end{proof}

There exist Borel probability spaces that do not satisfy Lebesgue's density
theorem (e.g. Example 5.6 in \cite{kaenmaki2010local}).

\bigskip

\begin{definition}
We say $(X,\mu)$ is \textbf{Vitali (}or \textbf{satifies Vitali's covering
theorem) }if for every Borel set $A\subset M$ and every $N,\varepsilon>0,$
there exists a finite subset $F\subset$ $A$ and $n_{x}\geq N$ (defined for
every $x\in F)$ such that $\left\{  B_{n_{x}}(x)\right\}  _{x\in F}$ are
disjoint and $\mu(A\diagdown\cup_{x\in F}B_{n_{x}}(x))\leq\varepsilon.$
\end{definition}

If $X\subset\mathbb{R}^{d}$ and $\mu$ is a Borel probability measure then
$(X,\mu)$ is Vitali (\cite{heinonen2001lectures} pg.8). This result is known
as Vitali's covering theorem.

Using a clopen base it is not hard to see that if $X$ is a Cantor space and
$\mu$ a Borel probability measure then $(X,\mu)$ satisfies Vitali's covering
theorem. For more conditions see Chapter 1 in \cite{heinonen2001lectures}.

The following theorem will be proved at the end of the following subsection.

\begin{theorem}
\label{3i1}Let $(X,T)$ be a TDS and $\mu$ a Borel probability measure.
Consider the following properties:

1) $(X,T)$ is $\mu-$equicontinuous.

2) Almost every $x\in X$ is a $\mu-$equicontinuity point.

3) There exists $X^{\prime}\subset X$ such that $\mu(X^{\prime})=1$ and
$(X^{\prime},d_{0})$ is separable

We have that $1)\Longleftrightarrow3.$

If $(X,\mu)$ is Vitali then $2)\Rightarrow1).$

If $(X,\mu)$ satisfies Lebesgue's density theorem then $1)\Rightarrow2).$
\end{theorem}

A $\mathbb{G-}$\textbf{topological dynamical system (}$\mathbb{G-}%
$\textbf{TDS)} is a pair $(X,T),$ where $X$ is a compact metric space,
$\mathbb{G}$ a semigroup and $T:=\left\{  T^{i}:i\in\mathbb{G}\right\}  $ is a
$\mathbb{G}-$ continuous action on $X.$ We can similarly define the orbit
metric as $d_{o}(x,y):=\sup_{i\in\mathbb{G}}\left\{  d(T^{i}x,T^{i}y)\right\}
.$ With this it is possible to define $\mu-$equicontinuity and $\mu
-$equicontinuity points. Theorem \ref{3i1} also holds on this set up. The
proof is identical. 

\subsection{$\mu-$Continuity}

\begin{definition}
\label{ft}Given a TDS $(X,T)$ we denote the change of metric projection map as
$f_{T}:(X,d)\rightarrow(X,d_{o})$.
\end{definition}

The function $f_{T}^{-1}$ is always continuous. A point $x\in X$ is an
equicontinuity point of $(X,T)$ if and only if it is a continuity point of
$f_{T}.$ Thus $(X,T)$ is equicontinuous if and only if $f_{T}$ is continuous.

\begin{definition}
Let $X,Y$ be metric spaces and $\mu$ a Borel probability measure on $X.$ A
function $f:X\rightarrow Y$ is $\mu-$\textbf{Lusin (}or \textbf{Lusin
measurable}) if for every $\varepsilon>0$ there exists a compact set $M\subset
X$ such that $\mu(M)>1-\varepsilon$ and $f\mid_{M}$ is continuous.
\end{definition}

This implies that a TDS $(X,T)$ is $\mu-$equicontinuous if and only if $f_{T}$
is $\mu-$Lusin.

We will define a measure theoretic notion of continuity point ($\mu
-$continuity point) that satifies the following property: $x\in X$ is a $\mu
-$equicontinuity point of $T$ if and only if $x$ is a $\mu-$continuity point
of $f_{T}$. We will show that if $(X,\mu)$ is Vitali and almost every $x\in X$
is a $\mu-$continuity point of $f$ then $f$ is $\mu-$Lusin; we show the
converse is true if $(X,\mu)$ satisfies Lebesgue's density theorem (Theorem
\ref{tfae}).

In this subsection $X$ will denote a compact metric space, $\mu$ a Borel
probability measure on $X,$ and $Y$ a metric space with metric $d_{Y}$ (and
balls $B_{m}^{Y}(y)\mathbf{:}=\left\{  z\in Y:d_{Y}(y,z)\leq1/m\right\}  )$.

\begin{definition}
In some cases we can talk about the measure of not necessarily measurable
sets. Let $A\subset X.$ We say $A$ has\textbf{ full measure} if $A$ contains a
measurable subset with measure one. We say $\mu(A)<\varepsilon,$ if there
exists a measurable set $A^{\prime}\supset A$ such that $\mu(A^{\prime
})<\varepsilon.$
\end{definition}

\begin{definition}
\label{mucont}A set $A\subset X$ is $\mu-$\textbf{measurable} if $A$ is in the
the completion of $\mu.$

The function $f:X\rightarrow Y$ is $\mu-$\textbf{measurable }if for every
Borel set $D\subset Y$, we have that $f^{-1}(D)$ is $\mu-$measurable.

A point $x\in X$ is a $\mu-$\textbf{continuity} \textbf{point} if for every
$m\in\mathbb{N}$ we have that $1_{f^{-1}(B_{m}^{Y}(f(x)))}$ is $\mu\times\mu
-$measurable and%
\begin{equation}
\lim_{n\rightarrow\infty}\frac{\mu\left[  f^{-1}(B_{m}^{Y}(f(x)))\cap
B_{n}(x)\right]  }{\mu(B_{n}(x))}=1.\label{limite}%
\end{equation}
If $\mu-$almost every $x\in X$ is a $\mu-$continuity point we say $f$ is
$\mu-$\textbf{continuous.}
\end{definition}

Every $\mu-$Lusin function is $\mu-$measurable. Lusin's theorem states the
converse is true if $Y$ is separable (note that $(X,d_{0})$ is not necessarily
separable). This fact is generalized in the following theorem.

\begin{theorem}
[\cite{simonnet1996measures} pg. 145]\label{simon}Let $f:X\rightarrow Y$ be a
function and $\mu$ a Borel probability measure on $X$ such that there exists
$X^{\prime}\subset X$ such that $\mu(X^{\prime})=1$ and $f(X^{\prime})$ is
separable. The function $f$ is $\mu-$Lusin if and only if for every ball $B,$
$f^{-1}(B)$ is $\mu-$measurable.
\end{theorem}

\begin{lemma}
\label{countable}Let $(X,\mu)$ be Vitali and $f$ :$X\rightarrow Y$ $\mu
-$continuous. For every $m\in\mathbb{N}$ and $\varepsilon>0$ there exists a
finite set $F_{m,\varepsilon}$ and a function $n(x)$ such that $\left\{
B_{n(x)}(x)\right\}  _{x\in F_{m,\varepsilon}}$ are disjoint, and $\mu
(\cup_{x\in F_{m,\varepsilon}}\left[  f^{-1}(B_{m}^{Y}(f(x)))\cap
B_{n(x)}(x)\right]  )>1-\varepsilon.$
\end{lemma}

\begin{proof}
Let $\varepsilon>0.$ For every $\mu-$equicontinuity point $x$ and
$m\in\mathbb{N}$ there exists $N_{x}$ such that $\frac{\mu(f^{-1}(B_{m}%
^{Y}(f(x)))\cap B_{n}(x))}{\mu(B_{n}(x))}>1-\varepsilon$ for every $n\geq
N_{x}.$ Let
\[
A_{i}:=\left\{  x\mid x\text{ is a }\mu-\text{continuity point, }N_{x}\leq
i\right\}  .
\]
Using the function $f=1_{f^{-1}(B_{m}^{Y}(f(x)))}\cdot1_{B_{n}(x)}$ and
Fubini's theorem one can check that $\frac{\mu(f^{-1}(B_{m}^{Y}(f(x)))\cap
B_{n}(x))}{\mu(B_{n}(x))}$ is a measurable function and hence the sets $A_{i}$
are also $\mu-$measurable. Since $\cup_{i\in\mathbb{N}}A_{i}$ contains the set
of $\mu-$continuity points, there exists $N$ such that $\mu(X\diagdown
A_{N})<\varepsilon.$ Since $(X,\mu)$ is Vitali there exists a finite set of
points $F_{m,\varepsilon}\subset$ $A_{N}$ and a function $n(x)\geq N$ such
that $\left\{  B_{n(x)}(x)\right\}  _{x\in F_{m,\varepsilon}}$ are disjoint
and $\mu(\cup_{x\in F_{m,\varepsilon}}B_{n(x)}(x))>\mu(A_{N})-\varepsilon.$
This implies $\mu(\cup_{x\in F_{m,\varepsilon}}B_{n(x)}(x))>1-2\varepsilon.$
Since $F_{m,\varepsilon}\subset A_{N},$ we obtain
\[
\mu(\cup_{x\in F_{m,\varepsilon}}\left[  f^{-1}(B_{m}^{Y}(f(x)))\cap
B_{n(x)}(x)\right]  )>1-3\varepsilon.
\]

\end{proof}

\begin{theorem}
\label{tfae}Let $f:X\rightarrow Y$. Consider the following properties

1) $f$ is $\mu-$Lusin.

2) $f$ is $\mu-$continuous.

If $(X,\mu)$ is Vitali then $2)\Longrightarrow1)$.

If $(X,\mu)$ satisfies Lebesgue's density theorem then $1)\Longrightarrow2).$
\end{theorem}

\begin{proof}
$2)+Vitali\Longrightarrow1)$

Let $m\in\mathbb{N}$ and $\varepsilon>0.$ By Lemma \ref{countable} there
exists a finite set of points $F_{m,\varepsilon/2^{m}}$ and a function $n_{x}$
(defined on $F_{m,\varepsilon/2^{m}}$) such that
\[
\mu(\cup_{x\in F_{m,\varepsilon/2^{m}}}\left[  f^{-1}(B_{m}^{Y}(f(x)))\cap
B_{n_{x}}(x)\right]  )>1-\varepsilon/2^{m}%
\]
and $\left\{  B_{n(x)}(x)\right\}  _{x\in F_{m,\varepsilon/2^{m}}}$ are
disjoint$.$ Let%
\[
x_{1},x_{2}\in\cup_{x\in F_{m,\varepsilon/2^{m}}}G_{x}:=\cup_{x\in
F_{m,\varepsilon/2^{m}}}\left[  f^{-1}(B_{m}^{Y}(f(x)))\cap B_{n_{x}%
}(x)\right]  .
\]
If $x_{1}$ and $x_{2}$ are sufficiently close then they will be in the same
set $G_{x}$ and hence $f(x_{1})\in B_{2m}^{Y}(f(x_{2})).$

Let
\[
M:=\cap_{m\in\mathbb{N}}\cup_{x\in F_{m,\varepsilon/2^{m}}}\left[
f^{-1}(B_{m}^{Y}(f(x)))\cap B_{n_{x}}(x)\right]  .
\]
Then $f\shortmid_{M}$is continuous and $\mu(M)>1-\varepsilon$. The regularity
of the Borel measure gives us the existence of the compact set.

$1)+LDT\Rightarrow2)$

Since $f$ is $\mu-$Lusin we have that it is $\mu-$measurable.

Let $\varepsilon>0.$ There exists a compact set $M\subset X$ such that
$\mu(M)>1-\varepsilon$ and $f\mid_{M}$ is continuous. Using Lebesgue's density
theorem we have%
\[
\lim_{n\rightarrow\infty}\frac{\mu(M\cap B_{n}(x))}{\mu(B_{n}(x))}=1\text{
\ for a.e. }x\in M.
\]

Let $m\in\mathbb{N}$. Since $f\mid_{M}$ is continuous then for sufficiently
large $n$ and almost every $x\in M,$ $M\cap B_{n}(x)\subset f^{-1}(B_{m}%
^{Y}(f(x))),$ and so for almost every $x$%
\[
\lim_{n\rightarrow\infty}\frac{\mu(f^{-1}(B_{m}^{Y}(f(x)))\cap B_{n}(x))}%
{\mu(B_{n}(x))}=1.
\]
This implies the set of $\mu-$continuity points has measure larger than
$1-\varepsilon$. We conclude the desired result.
\end{proof}

\begin{proof}
[Proof of Theorem \ref{3i1}]Take $f=f_{T}$\bigskip, $(Y,d_{Y})=(X,d_{o})$,
$B_{m}^{Y}=B_{m}^{o}.$

Note that orbit balls are countable intersections of balls, hence measurable.

$3)\Rightarrow1)$

Apply Theorem \ref{simon}.

$1)\Rightarrow3)$

If $f_{T}$ is $\mu-$Lusin, then that for every $\kappa>0$ there exists a
compact set $M_{\kappa}\subset X$ such that $\mu(M_{\kappa})>1-\kappa$ and
$f_{0}\mid_{M_{\kappa}}$ is continuous. This implies that $X^{\prime}%
=\cup_{n\in\mathbb{N}}M_{1/n}$ $\ $satisfies $\mu(X^{\prime})=1$ and
$(X^{\prime},d_{0})$ is separable.

For the other implications apply Theorem \ref{tfae}.
\end{proof}

The concept of approximate continuity for metric measure spaces has a similar
flavour to $\mu-$continuity. It also emulates locally the definition of
continuity using measures. The definition of an approximate continuity point
is stronger than (\ref{limite}) (in Definition \ref{mucont}) but it does not
assume the measurability of any set. A classical result states that if the
image of a function is separable then approximate continuity and measurability
are equivalent (Theorem 2.9.13 \cite{federer1996geometric}).

Similar notions as (\ref{limite}) have been studied. For example a result by
Sierpinski states that every function $f:\mathbb{R\rightarrow R}$ satisfies
(\ref{limite}) with respect to $\mu=\nu^{\ast},$ where $\nu^{\ast\text{ }}$is
the outer measure defined by Lebesgue's measure (Theorem 2.6.2
\cite{kannan1996advanced}).

\subsection{Local periodicity}

In this subsection we study local periodicity and $\mu-LEP$ systems, a
subclass of $\mu-$equicontinuous systems.

\begin{definition}
\label{leptds}Let $(X,T)$ be a TDS, $x\in X$ and $m\in\mathbb{N}$. We define
$LP_{m}(T)$ as the set of points $x$ such that there exists $p\in\mathbb{N}$
with
\[
d(T^{i}x,T^{i+jp}x)\leq1/m\text{\ \ \ \ for every }i,j\in\mathbb{Z}_{+}.
\]

We define $LEP_{m}(T)$ as the set of points, $x$ such that there exists
$p,q\in\mathbb{N}$ with
\[
d(T^{i}x,T^{i+jp}x)\leq1/m\text{ \ \ \ for every }i\geq q,j\in\mathbb{Z}_{+}.
\]
For $x\in LEP_{m}(T)$, $p_{m}(x)$ denotes the smallest possible $p$ and
$pp_{m}(x)$ the smallest possible $q$ for $p_{m}(x).$ We also define%
\begin{align*}
LP(T)  &  :=\cap_{m\in\mathbb{N}}LP_{m}(T)\\
LEP(T)  &  :=\cap_{m\in\mathbb{N}}LEP_{m}(T).
\end{align*}

A TDS $(X,T)$ is \textbf{locally eventually periodic} $(LEP)$ if $LEP(T)=X$
and\textbf{\ locally periodic}\emph{\ }$(LP)$ if $LP(T)=X.$
\end{definition}

\begin{remark}
\label{cantor}If $X$ is a Cantor space then $LP_{m}(T)$ is the set of points
$x$ such that $(T^{i}x)_{W_{m}}$ is periodic.
\end{remark}

\begin{example}
\label{odom}Let $S=(s_{0},s_{1},...)$ be a finite or infinite sequence of
natural numbers. The $S-$\textbf{adic odometer} is the $+(1,0,...)$ (with
carrying) map defined on the compact set $D=\prod\nolimits_{i\geq0}%
\mathbb{Z}_{s_{i}}$ (for a survey on odometers see \cite{surveyodometers}).
\end{example}

Odometers are locally periodic.

Subshifts are a special kind of TDS on symbolic spaces where $T$ is the left
shift, $\sigma$ (see \cite{lindmarcus} for definitions). It is not hard to see
that if $(X,\sigma)$ is a subshift then $x\in LP(\sigma)$ if and only if $x$
is right-periodic.

A closely related concept is regular recurrence. A point $x\in X$ is
\textbf{regularly recurrent} if for every $m$ there exists $p>0$ such that
$d(x,T^{ip}x)\leq1/m$ for every $i\in\mathbb{N}.$ Every $LP$ point is
regularly recurrent, but regularly recurrent points are not necessarily even
$LEP$ (for example Toeplitz subshifts \cite{surveyodometers}). $LEP$ points
are not necessarily regularly recurrent (the point 1000... in a one-sided
subshift). If every point of a minimal $TDS$ is regularly recurrent then it is
conjugate to an odometer \cite{Block2004151}, and hence $LP.$

\bigskip

Equicontinuity and local periodicity are related concepts.

Using results from \cite{bruckner} it is easy to show that if $X$ is the unit
interval and $(X,T)$ an equicontinuous TDS then $(X,T)$ is $LEP.$ Using
results from \cite{Valaristos1998} it is easy to see that if $X$ is the unit
circle and $(X,T)$ an equicontinuous TDS then $(X,T)$ is either an irrational
rotation or $LEP.$

Suppose $X$ a subshift and $T:X\rightarrow X$ a continuous shift-commuting
transformation (these TDS are known as cellular automata or shift
endomorphisms). In \cite{mueqca} it is shown that if $(X,T)$ is $LP$ then it
is equicontinuous, that if $(X,T)$ is equicontinuous then it is $LEP,$ and
that there exist non-equicontinuous $LEP$ systems.

\bigskip

\begin{definition}
Let $(X,T)$ be a TDS and $\mu$ a Borel probability measure on $X$. If
$\mu(LEP(T))=1$ we say $(X,T)$ is $\mu-$\textbf{locally eventually periodic}
($\mu-\mathbf{LEP})$ and $\mu$ is\textbf{\ }$T-$\textbf{locally eventually
periodic}\emph{\ (}$\mathbf{T-LEP}).$ We define $\mu-\mathbf{LP}$ and
$\mathbf{T-LP}$ analogously.
\end{definition}

\begin{lemma}
\label{definitiony}Let $m\in\mathbb{N}$ and $\varepsilon>0.$ If $(X,T)$ is
$\mu-LEP$ then there exist positive integers $p_{\varepsilon}^{m}$ and
$q_{\varepsilon}^{m}$ such that $\mu(Y_{\varepsilon}^{m})>1-\varepsilon,$
where
\begin{equation}
Y_{\varepsilon}^{m}:=\left\{  x\mid x\in LEP(T),\text{ with }p_{m}(x)\leq
p_{\varepsilon}^{m}\text{ and }pp_{m}(x)\leq q_{\varepsilon}^{m}\right\}  .
\label{laye}%
\end{equation}

\end{lemma}

\begin{proof}
Let
\[
Y:=\cup_{s,t\in\mathbb{N}}\left\{  x\mid x\in LEP(T),\text{ with }p_{m}(x)\leq
s\text{ and }pp_{m}(x)\leq t\right\}  .
\]
Since $(X,T)$ is $\mu-LEP$ we have that $\mu(Y)=1.$ Monotonicity of the
measure gives the desired result.
\end{proof}

\begin{definition}
Let $(X,T)$ be a $\mu-LEP$ TDS, $m\in\mathbb{N}$, and $\varepsilon>0.$ We will
use $p_{\varepsilon}^{m}$ and $q_{\varepsilon}^{m}$ to denote a particular
choice of integers that satisfy the conditions of Lemma \ref{definitiony} and
that satisfy $p_{\varepsilon}^{m}\rightarrow\infty$ and $q_{\varepsilon}%
^{m}\rightarrow\infty,$ as $\varepsilon\rightarrow0.$ We define
$Y_{\varepsilon}^{m}$ as in (\ref{laye}).
\end{definition}

\begin{proposition}
\label{lepeq}Let $(X,T)$ be a TDS. If $(X,T)$ is $\mu-LEP$ then $(X,T)$ is
$\mu-$equicontinuous.
\end{proposition}

\begin{proof}
Let $\varepsilon>0$ and $M:=\cap_{m\in\mathbb{N}}Y_{\varepsilon/2^{m}}^{m}$
(hence $\mu(M)>1-\varepsilon).$

Let $m\in\mathbb{N}$ and $x\in Y_{\varepsilon/2^{m}}^{m}.$ There exists $K$
such that if $d(x,y)<1/K$ then $d(T^{i}x,T^{i}y)<1/m$ for $0\leq i\leq\left(
p_{\varepsilon}^{m}\right)  ^{2}+q_{\varepsilon}^{m}.$ Let $x,y\in
Y_{\varepsilon}^{m}$, $d(x,y)<1/K$ , $p=p_{m}(x)p_{m}(y)$ and $i\geq
p+q_{\varepsilon}^{m}.$ We can express $i=q_{\varepsilon}^{m}+j\cdot p+k$ with
$j\in\mathbb{N}$ and $k\leq p.$ Using the fact that $x,y\in Y_{\varepsilon
}^{m}$ and $q_{\varepsilon}^{m}+k\leq\left(  p_{\varepsilon}^{m}\right)
^{2}+q_{\varepsilon}^{m}$ we obtain
\begin{align*}
d(T^{i}x,T^{i}y)  &  \leq d(T^{i}x,T^{q_{\varepsilon}^{m}+k}%
x)+d(T^{q_{\varepsilon}^{m}+k}x,T^{q_{\varepsilon}^{m}+k}%
y)+d(T^{q_{\varepsilon}^{m}+k}y,T^{i}y)\\
&  \leq3/m.
\end{align*}

This means that $B_{K}(x)\cap Y_{\varepsilon}^{m}\subset B_{m/3}^{o}(x)$ and
hence $T\mid_{M}$ is equicontinuous.
\end{proof}

The converse of this proposition is not true in general. An irrational
rotation on the circle is equicontinuous, and hence $\mu-$equicontinuous, but
contains no $LEP$ points so it is not $\mu-LEP.$ For cellular automata there
are conditions (for example if $X$ is a shift of finite type and $\mu$ a
shift$-$invariant measure) under which $\mu-$equicontinuity implies $\mu-LEP$
\cite{mueqca}.

\section{Invariant measures and spectral properties}

\bigskip In this section we are interested in measure preserving topological
dynamical systems.

We say $(M,T,\mu)$ is a \textbf{measure preserving transformation (MPT)} if
$(M,\mu)$ is a probability measure space, and $T:M\rightarrow M$ is a measure
preserving transformation (measurable and $T\mu=\mu$)$.$ We say $(M,T,\mu)$ is
\textbf{ergodic} if it is a MPT and every invariant set has measure 0 or 1.

Two measure preserving transformations $(M_{1},T_{1},\mu_{1})$ and
$(M_{2},T_{2},\mu_{2})$ are \textbf{ isomorphic (measurably)} if there exists
1-1 bi-measurable preserving function\textit{\ }$f:(X_{1},\mu_{1}%
)\rightarrow(X_{2},\mu_{2}),$ and that satisfies $T_{2}\circ f=f\circ T_{1}.$

The spectral theory of ergodic systems is particularly useful for studying
rigid transformations. We will give the definitions and state the most
important results. For more details and proofs see
\cite{walters2000introduction}.

A measure preserving transformation $T$ on a probability measure space
$(M,\mu)$ generates a unitary linear operator on the Hilbert space
$L^{2}(M,\mu),$ by $U_{T}:f\mapsto f\circ T,$ known as the \textbf{Koopman
operator}$.$ The spectrum of the Koopman operator is called the
\textbf{spectrum}\textit{\ }of the measure preserving transformation. We say
the spectrum of $(M,T,\mu)$ is\textbf{\ discrete (or pure point)} if there
exists an orthonormal basis for $L^{2}(M,\mu)$ which consists of
eigenfunctions of the Koopman operator$.$ The spectrum is\textit{\ }%
\textbf{rational} if the eigenvalues are complex roots of unity. Classical
results by Halmos and Von Neumann state that two ergodic MPT with discrete
spectrum have the same group of eigenvalues if and only if they are
isomorphic, and that an ergodic MPT has discrete spectrum if and only if it is
isomorphic to a rotation on a compact metric group.

\begin{theorem}
[\cite{mtequicontinuity}]Let $(X,\mu,T)$ be an ergodic $\mu-$equicontinuous
TDS. Then $(X,\mu,T)$ has discrete spectrum.
\end{theorem}

The converse is not true. For example Sturmian subshifts with their unique
invariant measure have discrete spectrum and it is not hard to see they have
no $\mu-$equicontinuity points. Nonetheless using a weaker version of $\mu
-$equicontinuity ($\mu-$mean equicontinuity) it is possible to characterize
discrete spectrum for ergodic TDS (see \cite{weakeq}).

In the next subsection we show that ergodic $\mu-LEP$ TDS on\ Cantor spaces
have rational spectrum.

\subsection{Local periodicity}

\begin{lemma}
\label{limlep}Let $(X,T)$ be a $\mu-LEP$ TDS and $m\in\mathbb{N}$. If $x\in
LEP_{m}(T)\diagdown LP_{m}(T)$ then
\[
\lim_{n\rightarrow\infty}T^{n}\mu(B_{m}^{o}(x))=0.
\]

\end{lemma}

\begin{proof}
Let $\varepsilon>0$, and $x\in LEP_{m}(T)\diagdown LP_{m}(T)$. The periods in
$Y_{\varepsilon}^{m}$ are bounded so if $n$ is sufficiently large then
$T^{-n}(B_{m}^{o}(x))\cap Y_{\varepsilon}^{m}=\emptyset;$ hence%
\[
\lim_{n\rightarrow\infty}T^{n}\mu(B_{m}^{o}(x))<\varepsilon.
\]

\end{proof}

\begin{proposition}
\label{invlp}Let $(X,T)$ be a $\mu-LEP$ TDS. If $T\mu=\mu$ then $T$ is
$\mu-LP.$
\end{proposition}

\begin{proof}
Since $(X,T)$ is $\mu-$equicontinuous we have there exists $X^{\prime}\subset
X$ such that $X^{\prime}$ is $d_{o}-$separable and $\mu(X^{\prime})=1$
(Theorem \ref{3i1})$.$ If $x\in LEP(T)\diagdown LP(T)$ then by Lemma
\ref{limlep} there exists $m$ such that $T^{n}\mu(B_{m}^{o}(x))\rightarrow$
$0.$ Using the fact that $\mu$ is an invariant measure we get that $\mu
(B_{m}^{o}(x))=T^{n}\mu(B_{m}^{o}(x))=0.$ This implies $\mu(X^{\prime}\cap
LEP(T)\diagdown LP(T))=0.$ Therefore $\mu(LP(T))=1.$
\end{proof}

\begin{theorem}
Let $X$ be a Cantor space, $(X,T)$ a TDS and $\mu$ an invariant probability
measure. If $(X,T)$ is $\mu-LEP$ then $(X,T,\mu)$ has discrete rational spectrum.
\end{theorem}

\begin{proof}
Using Proposition \ref{invlp} we have that $(X,T)$ is $\mu-LP.$

Let $m\in\mathbb{N},$ $y\in LP(T)$ and $i=\sqrt{-1}.$ We define $\lambda
_{m,y}:=e^{2\pi i/p_{m}(y)}\in\mathbb{C}$ and
\[
f_{m,y,k}:=\sum_{j=0}^{p_{m}(y)-1}\lambda_{m,y}^{j\cdot k}\cdot1_{B_{m}%
^{o}(T^{j}y)}\in L^{2}(X,\mu).
\]
Using the fact that $(X,T)$ is $\mu-LP,$ $y\in LP(T)$ and $X$ is a Cantor
space (see Remark \ref{cantor}) we have that $B_{m}^{o}(T^{p_{m}(y)}%
y)=B_{m}^{o}(y)$ and%
\[
U_{T}1_{B_{m}^{o}(T^{j}y)}=\left\{
\begin{array}
[c]{cc}%
1_{B_{m}^{o}(T^{j-1}y)}\text{ \ } & \text{if }1\leq j<p_{m}(y)-1\\
1_{B_{m}^{o}(T^{p_{m}(y)-1}y)}\text{ } & \text{if }j=0
\end{array}
.\right.
\]
%

\begin{figure}[ptb]%
\centering
\caption{An LEP point on a Cantor set. }%
\includegraphics[
natheight=4.990000in,
natwidth=6.833800in,
height=3.5129in,
width=4.8049in
]%
{lep.tif}%
\end{figure}

This implies that
\begin{align*}
U_{T}f_{m,y,k} &  =\sum_{j=0}^{p_{m}(x)-1}\lambda_{m,y}^{j\cdot k}%
\cdot1_{B_{m}^{o}(T^{j-1}y)}\\
&  =\sum_{j=0}^{p_{m}(y)-1}\lambda_{m,y}^{(j+1)\cdot k}\cdot1_{B_{m}^{o}%
(T^{j}y)}\\
&  =\lambda_{m,y}^{k}f_{m,y,k}%
\end{align*}
Thus $f_{m,y,k}$ is an eigenfunction corresponding to the eigenvalue
$\lambda_{m,y}^{k},$ which is a complex root of unity$.$

Considering that
\[
\sum_{k=0}^{p_{m}(y)-1}\lambda_{m,y}^{j\cdot k}=\left\{
\begin{array}
[c]{cc}%
0 & \text{if }j>0\\
p_{m}(y) & \text{if }j=0
\end{array}
\right.  ,
\]
we obtain
\[
\frac{1}{p_{m}(y)}\sum_{k=0}^{p_{m}(y)-1}f_{m,y,k}=1_{B_{m}^{o}(y)}.
\]
This means $1_{B_{m}^{o}(y)\text{ }}\in Span\left\{  f_{m,y,k}\right\}  _{k}.$

Let $n\in\mathbb{N},$ and $x\in X.$ Since $(X,T)$ is $\mu-$ LP and $X$ is a
Cantor set there exists a sequence $\left\{  y_{i}\right\}  \subset LP(T)$
such that $B_{n}^{o}(y_{i})$ are disjoint and $\mu(B_{n}(x))=$ $\mu(\cup
B_{n}^{o}(y_{i})).$ This means $1_{B_{n}(x)}$ can be approximated in $L^{2}$
by elements in
\[
Span\left\{  1_{B_{m}^{o}(y)}:m\in\mathbb{N}\text{ and }y\in LP(T)\right\}  .
\]
Since the closure of $Span\left\{  1_{B_{n}(x)}:n\in\mathbb{N}\text{, }x\in
X\right\}  $ is $L^{2}(X,\mu)$ we conclude the closure of $Span\left\{
f_{m,y,k}\right\}  _{m,y,k}$ is $L^{2}(X,\mu).$
\end{proof}

\begin{corollary}
Let $X$ be a Cantor space, $(X,T)$ a TDS and $\mu$ an ergodic probability
measure. If $(X,T)$ is $\mu-LEP$ then $(X,T,\mu)$ is isomorphic to an odometer
\end{corollary}

\begin{proof}
The set of odometers cover all the possible rational spectrums we obtain the
following corollary. Two MPT with discrete spectrum are isomorphic if and only
if they are spectrally isomorphic.
\end{proof}

This result is particularly useful for cellular automata, as $\mu-LEP$
invariant measures appear naturally as the limit measures of $\mu
-$equicontinuous CA; for more information see \cite{mueqca}.

\subsection{$\mu-$Sensitivity}

The following notion of measure theoretical sensitivity was introduced in
\cite{Cadre2005375}, and it was characterized in \cite{mtequicontinuity}. We
provide another characterization.

\begin{definition}
[\cite{Cadre2005375}\cite{mtequicontinuity}]Let $(X,T)$ be a TDS and $\mu$ an
invariant measure. We define the set $S(\varepsilon):=\left\{  (x,y)\in
X^{2}:\exists n>0\text{ such that }d(T^{n}x,T^{n}y)\geq\varepsilon\right\}  .$

We say $(X,T)$ is $\mu-$\textbf{sensitive} (or $\mathbf{\mu-}$\textbf{pairwise
sensitive}) if there exists $\varepsilon>0$ such that $\mu\times
\mu(S(\varepsilon))=1.$
\end{definition}

In \cite{mtequicontinuity} it was shown that if $\mu$ is ergodic then either
$(X,T)$ is $\mu-$sensitive or $\mu-$equicontinuous.

\begin{theorem}
\label{musens}Let $(X,\mu)$ be a Vitali space that satisfies Lebesgue's
density theorem $(X,T)$ be a TDS and $\mu$ an ergodic measure. Then $(X,T)$ is
$\mu-$sensitive if and only if $X$ contains no $\mu-$equicontinuity points.
\end{theorem}

\begin{proof}
Let $E_{\mu}$ be the set of $\mu-$equicontinuity points.

By Theorem \ref{3i1} and the previous comment we have that $T$ is $\mu
-$sensitive or $E_{\mu}$ has measure one.

If $E_{\mu}=\varnothing$ then $(X,T)$ is $\mu-$sensitive.

On the other hand note that if $x\in E_{\mu}$ then for all $\varepsilon=1/m>0$
we have that $\mu(B_{m}^{o}(x))>0.$ This means that $\mu\times\mu
(S(\varepsilon))<1$ for all $\varepsilon>0.$
\end{proof}

\bibliographystyle{aabbrv}
\bibliography{camel}

\end{document}